\documentclass[12pt,a4paper]{article}

\usepackage[latin1]{inputenc}
\usepackage{amsmath,amsfonts,amssymb,amsthm,xcolor,tikz}
\usepackage{hyperref}
\usepackage[scale=0.8]{geometry}

\newtheorem{lemma}{Lemma}
\newtheorem{theorem}{Theorem}

\DeclareMathOperator{\re}{Re}
\DeclareMathOperator{\im}{Im}
\DeclareMathOperator{\interior}{int}
\DeclareMathOperator{\diag}{diag}
\DeclareMathOperator{\bez}{Bez}

\title{Finite Blaschke products with prescribed critical points,
	Stieltjes polynomials, and moment problems}
\author{Gunter Semmler and Elias Wegert}
\date{}

\begin{document}

\maketitle


\begin{abstract}
The determination of a finite Blaschke product from its critical
points is a well-known problem with interrelations to several other
topics. Though existence and uniqueness of solutions are established
for long, we present new aspects which have not yet been explored to
their full extent. In particular, we show that the following three
problems are equivalent: (i) determining a finite Blaschke product
from its critical points, (ii) finding the equilibrium position of
moveable point charges interacting with a special configuration of
fixed charges, and (iii) solving a moment problem for the canonical
representation of power moments on the real axis. These equivalences
are not only of theoretical interest, but also open up new
perspectives for the design of algorithms. For instance, the second
problem is closely linked to the determination of certain Stieltjes
and Van Vleck polynomials for a second order ODE and characterizes
solutions as global minimizers of an energy functional.

\end{abstract}


\section{Introduction} \label{sec0}

A \emph{finite Blaschke product} of degree $n$  is a rational
function of the form
\[
B(z)=c\prod_{k=1}^n\frac{z-a_k}{1-\overline{a}_kz}
\]
whose (not necessarily distinct) zeros $a_1,\dots, a_n$ are in the unit disc $\mathbb D$ and
$c\in \mathbb T:=\partial \mathbb D$. A point $\xi$ where
$B'(\xi)=0$ is called \emph{critical point} of the Blaschke product.
Every Blaschke product of degree $n$ has exactly $n-1$ critical points
$\xi_k$ in $\mathbb D$ (counted with multiplicities), and another
$n-1$ critical points $1/\overline{\xi}_k$, symmetric to $\xi_k$
with respect to the unit circle $\mathbb T$ (see e.g. \cite{Mash}). Note that if $a_j$ is a $k$-fold zero then $a_j$ and $1/\overline{a}_j$ are both   criticals points of order $k-1$, and if $0$ is a critical point then we  also count  $\infty$ as a critical point of the same multiplicity.

While it is straightforward to determine the critical points of
a Blaschke product from its zeros (by computing the zeros of a polynomial),
the reverse problem is much more challenging.
The basic existence and uniqueness result is summarized in the following
theorem.

\begin{theorem}\label{Th1}
Let $\xi_1,\dots, \xi_{n-1}$ be $n-1$  points in
$\mathbb D$. Then there is a Blaschke product $B$ of degree $n$ with
critical points $\xi_k$. $B$ is unique up to post-composition with a
conformal automorphism of $\mathbb D$.
\end{theorem}

This theorem has been proved (in chronological order) by Heins~\cite{hei},
Wang and Peng~\cite{wan}, Bousch~\cite{bou}, and Zakeri~\cite{zak},
using topological arguments. Stephenson \cite[Theorem 21.1]{ste} obtains
$B$ as the limit of discrete finite Blaschke products, i.e., he considers
sequences of  circle packings with prescribed branch set. Kraus and
Roth~\cite{kr}, \cite{kr2} describe an approach based on a solution of
the Gaussian curvature equation (that works equally well for infinite Blaschke
products), and ask for a procedure to actually compute $B$ from its critical
points.

In this paper we show that the determination of Blaschke products
with prescribed critical points is equivalent to two other classical
problems of analysis. The first one is obtained after transforming
the problem from the unit disk to the upper half plane in form of a
second order ODE. We will have to look for its polynomial solutions
known as \emph{Stieltjes} and \emph{Van Vleck polynomials}. Like in the
case originally considered by Stieltjes this allows an electrostatic
interpretation and the characterization of solutions
as minimum points of an energy functional. This approach yields a new
and (as we hope) very transparent proof of Theorem~\ref{Th1}.
Moreover, the polynomial encoding the
given critical points is positive on the real axis and therefore an
inner point of the convex cone of positive polynomials on $\mathbb R$.
By mapping it to an inner point of the cone of power moments we
demonstrate that the problem is also equivalent to the classical
problem of finding canonical representations of given moments.

Algorithmic aspects of the different approaches and the results of
numerical experiments will be discussed in a forthcoming paper.


\section{Transformations} \label{sec1}

Let $B$ be Blaschke product of degree $n\ge 2$.  As
in~\cite[Lemma~3]{sewe} it is convenient to transform the problem
using the (inverse) Cayley transform
\[
T(z)=\mathrm{i}\,\frac{1+z}{1-z},
\]
which maps $\mathbb D$ and $\mathbb T$ onto the upper half-plane
$\mathbb H:=\{x+\mathrm{i}y:y>0\}$ and the extended real line
$\mathbb R\cup\{\infty\}$, respectively, and satisfies $T(1)=\infty$,
$T(-1)=0$.
Consequently, $f:=T\circ B\circ T^{-1}$ is a rational function that
is real-valued on $\mathbb R$ (except at its poles). The transition between the unit disc and the upper half plane was also a key tool in the  work of Gorkin and Rhoades \cite{gorh} on boundary interpolatioin by finite Blaschke products. 

Let us first assume that $B$ satisfies the normalization $B(1)=-1$.
Then $f$ has a zero at infinity and therefore $f=p/q$ with
real polynomials $p$ of degree $n-1$ and $q$ of degree $n$. Since,
for an appropriate branch of the argument function, the mapping
$\tau\mapsto\arg B(\mathrm{e}^{\mathrm{i}\tau})$, is continuous
and strictly monotone from
$[0,2\pi)$ onto some interval $[\varphi, \varphi+2n\pi)$, $f$  has
$n$ simple poles at real numbers  $x_1<x_2<\dots<x_n$ corresponding
to the $n$ values of $\tau$ with $B(\mathrm{e}^{\mathrm{i}\tau})=1$.
In each of the intervals $(x_k, x_{k+1})$ as well as in
$(-\infty, x_1)$ and $(x_n,\infty)$, the function $f$ is strictly
increasing. Therefore, the partial fraction decomposition of $f$ has
the form
\begin{equation}\label{eq1}
f(x)=-\frac{r_1}{x-x_1}-\frac{r_2}{x-x_2}-\dots -\frac{r_n}{x-x_n}
\end{equation}
with positive numbers $r_k$. Conversely, if $f$ is a rational function of
the form (\ref{eq1}) with ordered poles $x_1<x_2<\dots<x_n$ and
$r_k>0$, then $B:=T^{-1}\circ f\circ T $ is a rational function that
maps $\mathbb D$ and $\mathbb T$ onto themselves and satisfies $B(1)=-1$.
Hence $B$ is a Blaschke product. By the argument principle it has exactly
$n$ zeros  in $\mathbb D$, i.e., $B$ has degree $n$. Thus we have shown:

\begin{lemma} \label{l1} 
The mapping $B\mapsto f:=T\circ B\circ T^{-1}$ is a bijection
between all Blaschke products $B$ of degree $n$ satisfying  $B(1)=-1$
and all rational functions $f$ of the form {\rm (\ref{eq1})} with
ordered poles $x_1<x_2<\dots<x_n$ and numbers $r_k>0$ for
$k=1,2,\dots, n$.
\end{lemma}

It will turn out useful to also consider Blaschke products with the
side condition $B(1)=1$. Then $g:=T\circ B\circ T^{-1}$ has a pole
at infinity and can be written in the form $g=p/q$ with real
polynomials $p$ of degree $n$ and $q$ of degree $n-1$, respectively.
As in the case  considered above, one derives the existence of
$n-1$ finite poles $t_1<\dots <t_{n-1}$  corresponding to the $n-1$
points $\tau\in (0,2\pi)$, where $B(\mathrm{e}^{\mathrm{i}\tau})=1$.
In each of the intervals $(t_k, t_{k+1})$ as well as in $(-\infty, t_1)$
and $(t_{n-1},\infty)$ the function $g$ is strictly increasing,
hence the partial fraction decomposition of $g$ has the form
\begin{equation}\label{eq11}
g(x)=ax+b-\frac{s_1}{x-t_1}-\dots-\frac{s_{n-1}}{x-t_{n-1}}
\end{equation}
with $a>0$, $b\in\mathbb{R}$ and $s_k>0$. Proceeding as above, we get
the following lemma.

\begin{lemma}\label{l2} 
The mapping $B\mapsto g:=T\circ B\circ T^{-1}$ is a bijection
between all Blaschke products $B$ of degree $n$ satisfying  $B(1)=1$
and all rational functions $g$ of the form {\rm (\ref{eq11})} with
$a>0$, $b\in \mathbb{R}$, $t_1<t_2<\dots<t_{n-1}$ and $s_k>0$ for
$k=1,2,\dots, n-1$.
\end{lemma}

We first investigate the problem for functions of the form (\ref{eq1}).
It follows from the chain rule that $\xi_k\in \mathbb D$ is a
critical point of $B$ if and only if $\zeta_k:=T(\xi_k)\in\mathbb
H$ is a critical point of $f$. Since M\"obius maps preserve
symmetries in circles and lines, also the points
$\overline{\zeta}_k$ are critical points of $f$ corresponding to the
critical points $1/\overline{\xi}_k$ of $B$.  The derivative of $f$
has the form
\begin{equation}\label{eq2}
f'(x)= \frac{r_1}{(x-x_1)^2}+\frac{r_2}{(x-x_2)^2}+\dots
+\frac{r_n}{(x-x_n)^2}
=\frac{cP(x)}{Q(x)^2}
\end{equation}
with a monic real polynomial $P$ of degree $2n-2$, $c:=r_1+r_2+\dots
+r_n$ and
\begin{equation}\label{eq12}
Q(x):=\prod\limits_{k=1}^n(x-x_k).
\end{equation}
We conclude that $P$ has the factorization
\begin{equation}\label{eq13}
P(x)=   \prod_{k=1}^{n-1}(x-\zeta_k)(x-\overline{\zeta}_k)
=\prod_{k=1}^{n-1}\big((x-a_k)^2+b_k^2\big)
\end{equation}
where $a_k:=\mathrm{Re}\,\zeta_k$, $b_k:=\mathrm{Im}\,\zeta_k$.
The polynomial $P$ is entirely determined by the location of the critical points
and \eqref{eq2} shows that it satisfies the equation
\begin{equation}\label{eq3}
cP(x)=\sum_{k=1}^n r_k\prod_{\substack{j=1\\j\neq k}}^n (x-x_j)^2.
\end{equation}
Evaluating {\rm\eqref{eq3}} at $x=x_k$ we get
\begin{equation}\label{eq7}
r_k=\frac{cP(x_k)}{\prod\limits_{\substack{j=1\\j\neq k}}^n (x_k-x_j)^2},
\qquad k=1,2,\dots, n.
\end{equation}
Inserting  (\ref{eq7}) into (\ref{eq3}), and introducing the
\emph{Lagrange interpolation polynomials}
\begin{equation}\label{lagrange}
Q_k(x):=\prod_{\substack{j=1\\j\neq k}}^n \frac{x-x_j}{x_k-x_j}
=\frac{Q(x)}{Q'(x_k)(x-x_k)}, \qquad k=1,2,\dots, n,
\end{equation}
which satisfy $Q_k(x_j)=\delta_{kj}$, we can rewrite equation (\ref{eq3}) as
\begin{equation}\label{eq5}
P(x)=\sum_{k=1}^nP(x_k)Q_k(x)^2.
\end{equation}
Since $P$ has degree $2n-2$, the \emph{Lagrange-Hermite interpolation formula}
(cf.~Chapter~14.1 of \cite{sze}) implies that for \emph{any}  pairwise
distinct points $x_1,x_2,\dots, x_n$,
\begin{equation} \label{eq.star}
P(x)=\sum_{k=1}^nP(x_k)h_k(x)+\sum_{k=1}^n P'(x_k)\mathfrak{h}_k(x),
\end{equation}
where for $k=1,2,\dots, n$,
\begin{equation} \label{eq.plus}
h_k(x):=\left(1-\frac{Q''(x_k)}{Q'(x_k)}(x-x_k)\right)Q_k(x)^2,\qquad
\mathfrak{h}_k(x):=(x-x_k)Q_k(x)^2
\end{equation}
are the \emph{fundamental polynomials} of the first and second kind
of \emph{Hermite interpolation}. From formulas {\rm\eqref{eq.star}},
{\rm\eqref{eq.plus}} we get
\begin{equation}\label{eq6}
P(x)=\sum_{k=1}^nP(x_k)Q_k(x)^2+\sum_{k=1}^n\left(P'(x_k)-P(x_k)
\frac{Q''(x_k)}{Q'(x_k)}\right)(x-x_k)Q_k(x)^2.
\end{equation}
The polynomials $(x-x_k)Q_k(x)^2$ are linearly independent, and
comparing (\ref{eq5}) and (\ref{eq6}) we see that the representation
(\ref{eq5}) holds if and only if
\[
Q'(x_k)P'(x_k)=Q''(x_k)P(x_k), \qquad k=1,2,\dots, n.
\]
Since the $x_1,x_2,\dots, x_n$ are exactly the (simple) roots of $Q$,
this condition means that the polynomial $Q'P'-Q''P$ is divisible by $Q$,
i.e., there is a real  polynomial $R$ such that
\[
Q'P'-Q''P=QR.
\]
Because $Q'P'-Q''P$ has exact degree $3n-4$ we see that $R$ has degree $2n-4$.
One easily checks that these considerations can be reversed to obtain the
following equivalence statement.

\begin{lemma}\label{l3}
Let $n\ge 2$, $x_1<x_2<\dots<x_n$, $r_1,r_2,\dots, r_n>0$,
$\zeta_1,\zeta_2,\dots, \zeta_{n-1}\in \mathbb H$, and let $f$, $Q$,
$P$ be defined by {\rm (\ref{eq1})}, {\rm (\ref{eq12})}, {\rm
(\ref{eq13})}, respectively. Then the rational function  $f$ has critical
points $\zeta_1,\zeta_2,\dots, \zeta_{n-1}$ if and only if
there is a real polynomial $R$ of (exact) degree $2n-4$ such that
\begin{equation}\label{eq8}
PQ''-P'Q'+RQ=0,
\end{equation}
and the $r_k$ are given by \eqref{eq7} for some $c>0$.
\end{lemma}

The corresponding constructions for the function $g$ in (\ref{eq11})
are similar: Its derivative is
\begin{equation}\label{eq30}
g'(x)=a+\frac{s_1}{(x-t_1)^2}+\dots+\frac{s_{n-1}}{(x-t_{n-1})^2}
=\frac{a\,P(x)}{S(x)^2}.
\end{equation}
where $P$ is defined in {\rm\eqref{eq13}} as before, and
\begin{equation}\label{eq14}
S(x):=\prod_{k=1}^{n-1}(x-t_k).
\end{equation}
Now we  obtain the equation
\begin{equation} \label{eqAP}
a\,P(x)=a\prod_{k=1}^{n-1}(x-t_k)^2+\sum_{k=1}^{n-1}s_k
\prod_{\substack{j=1\\j\neq k}}^{n-1}(x-t_j)^2,
\end{equation}
in particular
\begin{equation}\label{eq4}
s_k=\frac{a\,P(t_k)}{\prod\limits_{\substack{j=1\\j\neq k}}^{n-1}(t_k-t_j)^2},
\qquad k=1,\dots, n-1.
\end{equation}
With the help of the Lagrange interpolation polynomials with nodes
$t_1, \dots, t_{n-1}$, given by
\begin{equation}
\label{lagrange2}
S_k(x):=\prod_{\substack{j=1\\j\neq k}}^{n-1} \frac{x-t_j}{t_k-t_j}
=\frac{S(x)}{S'(t_k)(x-t_k)}, \qquad k=1,2,\dots, n-1,
\end{equation}
equation {\rm\eqref{eqAP}} can be rewritten as
\begin{equation}\label{eq15}
P(x)=S(x)^2+\sum_{k=1}^{n-1}P(t_k)S_k(x)^2.
\end{equation}
Since $P$ and $S$ are monic, $P-S^2$ is at most of degree $2n-3$, and hence
the Lagrange-Hermite interpolation formula applied to this polynomial tells
us that
\[
P(x)-S(x)^2=\sum_{k=1}^{n-1}P(t_k)g_k(x)+\sum_{k=1}^{n-1} P'(t_k)\mathfrak{g}_k(x),
\]
where, for $k=1,2,\dots, n-1$,
\begin{equation*}
g_k(x):=\left(1-\frac{S''(t_k)}{S'(t_k)}(x-t_k)\right)S_k(x)^2,\qquad
\mathfrak{g}_k(x):=(x-t_k)\,S_k(x)^2.
\end{equation*}
A straightforward computation yields
\[
P(x)-S(x)^2=\sum_{k=1}^{n-1}P(t_k)S_k(x)^2
+\sum_{k=1}^{n-1}\left(P'(t_k)-P(t_k)
\frac{S''(t_k)}{S'(t_k)}\right)(x-t_k)\,S_k(x)^2,
\]
and comparing this with (\ref{eq15}) we get
\[
P'(t_k)S'(t_k)=P(t_k)S''(t_k), \qquad k=1,\dots, n-1,
\]
which is equivalent to the divisibility of $P'S'-PS''$ by $S$
(see {\rm\eqref{eq14}}).
The quotient $\tilde{R}$ is a polynomial of exact degree $2n-4$.
We summarize these results in the next lemma.

\begin{lemma} \label{l4} 
Let $n\ge 2$, $a>0$, $b\in\mathbb R$, $t_1<t_2<\dots<t_{n-1}$,
$s_1,s_2,\dots, s_{n-1}>0$, $\zeta_1,\zeta_2,\dots, \zeta_{n-1}\in
\mathbb H$, and let $g$,  $P$ and $S$ be given by {\rm (\ref{eq11})},
{\rm (\ref{eq13})}, and {\rm (\ref{eq14})}, respectively. Then the
rational function  $g$ has critical points $\zeta_1,\zeta_2,\dots,
\zeta_{n-1}$ if and only if there is a real polynomial $\tilde{R}$
of degree
$2n-4$ such that
\begin{equation}\label{eq16}
PS''-P'S'+\tilde{R}S=0,
\end{equation}
and the $s_k$ are given by \eqref{eq4}.
\end{lemma}

In this way the determination of a Blaschke product with given
critical points is reduced to the question for which real
polynomials $R$ (resp. $\tilde{R}$) the second order ODE (\ref{eq8})
(resp. (\ref{eq16})) has a polynomial solution $Q$ of degree $n$
(resp. $S$ of degree $n-1$) with simple real roots. This is a
classical problem which will be considered in the next section.


\section{Stieltjes and Van Vleck polynomials} \label{stieltjes}

Let $A$ and $B$ be given polynomials of degrees $p+1$ and $p$,
respectively. A polynomial $C$ of degree $p-1$ is called \emph{Van
Vleck polynomial}, if the \emph{generalized Lam\'e equation}
\begin{equation}\label{eq9}
AQ''+2BQ'+CQ=0
\end{equation}
has a polynomial solution $Q$ of preassigned degree $n$. The
solutions $Q$ are called \emph{Stieltjes polynomials} or
\emph{Heine-Stieltjes polynomials}. Under certain condition given
below, Stieltjes \cite{sti} proved the existence of the polynomials
that carry now his name (see also \cite[Section 6.8]{sze}).
Stieltjes assumed that
\[
A(x)=(x-a_0)(x-a_1)\cdots (x-a_p)
\]
has real roots $a_0<a_1<\dots <a_p$ and that the coefficients
$\varrho_k$  in the partial fraction decomposition
\[
\frac{B(x)}{A(x)}=\frac{\varrho_0}{x-a_0}+\frac{\varrho_1}{x-a_1}+\dots
+\frac{\varrho_p}{x-a_p}
\]
are all positive. Now consider  partitions $n=n_1+n_2+\dots+n_p$ of $n$
into $p$ non-negative integer summands $n_1,n_2,\dots,n_p$.
There are $\binom{n+p-1}{n}$ such partitions, and  each partition
corresponds to exactly one monic polynomial  $Q$ with $n$ distinct
real roots such that (\ref{eq9}) holds for a suitable Van Vleck
polynomial $C $. This polynomial $Q$ has exactly $n_k$
roots in $(a_{k-1},a_{k})$ for $k=1,2,\dots, p$ and can be
characterized as follows: Put positive charges $\varrho_k$  at the
fixed positions $a_k$ and movable unit charges on the real line such
that each interval $(a_{k-1}, a_k)$ contains exactly $n_k$ of them.
As usual in plane electrostatics, the force (repulsion or attraction)
between two charges is assumed to be proportional to their magnitudes
and to the inverse of their distance.
Stieltjes proved that there is a unique equilibrium
position of these $n$ positive movable unit charges; it is attained
when the charges are located at the zeros of $Q$ and corresponds to the
global minimum of the potential energy of each possible charge
configuration.

In our equations (\ref{eq8}) and (\ref{eq16}) the polynomial $P$
does not have real zeros so that Stieltjes' results are not directly
applicable. The necessary adaptations to the situation at hand will
be done in this section.

Assume first that $R$ is a Van Vleck polynomial for the equation
(\ref{eq8}) and that $Q$ has the form (\ref{eq12}). The identity
\begin{equation} \label{eqID}
Q''(x_k)=2\,Q'(x_k)\sum_{\substack{j=1\\j\neq k}}^n\frac{1}{x_k-x_j},
\qquad k=1,2,\dots, n,
\end{equation}
is well-known (see for instance formula (2.9) in \cite{ism}). Recall
that the logarithmic derivative of $P$ has the representation
\[
\frac{P'(x)}{P(x)}=\sum_{j=1}^{n-1}\left(\frac{1}{x-\zeta_j}
+\frac{1}{x-\overline{\zeta}_j}\right).
\]
Dividing \eqref{eq8} by $PQ'$ and evaluating at $x_k$ using {\rm\eqref{eqID}}
and $Q(x_k)=0$ we get
\begin{equation}\label{eq10}
\sum_{\substack{j=1\\j\neq k}}^n\frac{2}{x_k-x_j}
-\sum_{j=1}^{n-1}\left(\frac{1}{x_k-\zeta_j}
+\frac{1}{x_k-\overline{\zeta}_j}\right)=0,\qquad
k=1,2,\dots, n.
\end{equation}
Conversely, if the equations {\rm\eqref{eq10}} are satisfied,
the polynomial $P'Q'-PQ''$ has the zeros $x_1,\dots,x_n$, so that
$R:=(P'Q'-PQ'')/Q$ is a polynomial which satisfies {\rm\eqref{eq8}}.
Summarizing we get the following result:

\begin{lemma} \label{l5}
Let $\zeta_1, \dots, \zeta_{n-1}\in\mathbb H$ be given and define
$P$ by {\rm (\ref{eq13})}.  Then  the polynomial $Q$ from {\rm
(\ref{eq12})} with $x_1<\dots<x_n$ is a Stieltjes polynomial for
the equation {\rm (\ref{eq8})} if and only  if {\rm (\ref{eq10})}
holds. Similarly, the polynomial $S$ from {\rm (\ref{eq14})} with
$t_1<\dots<t_{n-1}$  is a Stieltjes polynomial for the equation
{\rm (\ref{eq16})} if and only if
\begin{equation}\label{eq17}
\sum_{\substack{j=1\\j\neq k}}^{n-1}\frac{2}{t_k-t_j}
-\sum_{j=1}^{n-1}\left(\frac{1}{t_k-\zeta_j}
+\frac{1}{t_k-\overline{\zeta}_j}\right)=0,\qquad
k=1,2,\dots, n-1.
\end{equation}
\end{lemma}

\noindent
Equation~(\ref{eq10}) has the following electrostatic interpretation:
Fix negative charges of size $-1/2$ at each of the
$2n-2$ points $\zeta_k$ and $\overline{\zeta}_k$. Then $n$ (moveable)
positive unit charges at the positions $x_1,\dots, x_n$ on the real
axis are in equilibrium in the field generated by all charges.
Analogously, (\ref{eq17}) describes the equilibrium
positions of $n-1$ positive unit charges at $t_1,\dots, t_{n-1}$ on
the real axis in the presence of the same $2n-2$ negative charges at
the points $\zeta_k, \overline{\zeta}_k$.

The equilibrium of moveable positive unit charges in the presence
of fixed negative charges in $\mathbb{C}\setminus \mathbb R$ was studied
by Orive and Garc\'{i}a \cite{org}. Their  results are applicable to
(\ref{eq17}), and we will include a proof only to make our exposition
self-contained (see Lemma \ref{l6} below). However, the equilibrium
problem described by (\ref{eq10}) is not covered by their results,
since the number of movable unit charges surpasses the total of negative
charges by $1$.

Other equilibrium problems involving fixed charges have recently been
considered by Gr\"un\-baum \cite{gru}, \cite{gru1},
Dimitrov and Van Assche \cite{diva}, \cite{dim2}, and Grinshpan \cite{gri};
see Marcell\'{a}n, Mart\'inez-Finkelshtein, and Mart\'inez-Gonz\'{a}lez
\cite{mmm} for an overview until 2007. More recent literature includes
Mart\'inez-Finkelshtein and Rakhmanov \cite{mara},
McMillen, Bourget, and Agnew \cite{mba},
Orive and S\'{a}nchez-Lara \cite{osl1}, \cite{osl}, and Shapiro \cite{shap}.
\medskip

{\bfseries Example 1.} Let $n=2$, i.e., there is only one critical point
$\zeta:=\zeta_1$ given in $\mathbb H$. Then (\ref{eq17}) with
$t:=t_1$ simplifies to
\[
\frac{1}{t-\zeta}+\frac{1}{t-\overline{\zeta}}=0.
\]
The only solution of this equation is $t=\re \zeta$, the expected
equilibrium position of one unit charge. The equations (\ref{eq10})
read as follows
\begin{equation} \label{eq18}
\frac{2}{x_1-x_2}-\left(\frac{1}{x_1-\zeta}
+\frac{1}{x_1-\overline{\zeta}}\right)=0,\qquad
\frac{2}{x_2-x_1}-\left(\frac{1}{x_2-\zeta}
+\frac{1}{x_2-\overline{\zeta}}\right)=0.
\end{equation}
From the first equation in (\ref{eq18}) we get
\[
2(x_1-\zeta)(x_1-\overline{\zeta})=((x_1-\overline{\zeta})+(x_1-\zeta))(x_1-x_2)
\]
and after some elementary manipulations we arrive at
\begin{equation}\label{eq20}
|x_1-\zeta|^2+|x_2-\zeta|^2=(x_1-x_2)^2.
\end{equation}
Since this equation is invariant with respect to interchanging $x_1$
and $x_2$, the second equation of (\ref{eq18}) yields the same conditon.
Equation (\ref{eq20}) has a simple geometric meaning; it says that, by Thales'
theorem, the points $\zeta, \overline{\zeta}, x_1$ and $x_2$ lie on a circle.

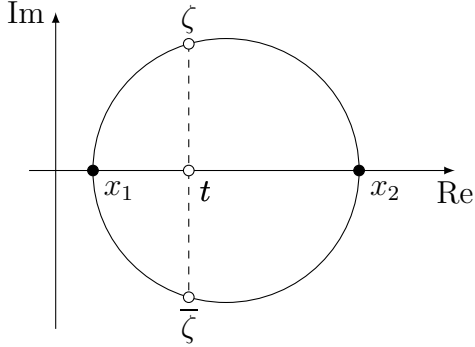
\begin{center}
\begin{tikzpicture}[>=latex, scale=0.7]
\draw[->] (-3,0)--(5,0)node[below]{$\re$};
\draw[->] (-2.5,-3)--(-2.5,3)node[left]{$\im$};
\draw[dashed] (0,12/5)--(0,-12/5);
\node[below right] at (0,0){$t$};
\draw (0.7,0)circle (2.5cm);
\draw[fill=white] (0,12/5)circle (1mm)node[above]{$\zeta$};
\draw[fill=white] (0,-12/5)circle (1mm)node[below]{$\overline{\zeta}$};
\draw[fill=black] (-9/5,0)circle (1mm)node[below right]{$x_1$};
\draw[fill=black] (16/5,0)circle (1mm)node[below right]{$x_2$};
\draw[fill=white] (0,0)circle (1mm)node[below right]{$t$};
\end{tikzpicture}
\end{center}

The equilibrium positions $x_1, x_2$ are therefore not uniquely determined:
for each $x_1\in(-\infty, t)$ there is a corresponding $x_2\in (t,\infty)$,
and vice versa.
As we will see later (Theorem~\ref{Th2}), the situation is similar for $n>2$.
\medskip

As pointed out in \cite{kr}, the \emph{uniqueness statement} in
Theorem~\ref{Th1} follows from Nehari's generalization of  Schwarz' Lemma
(see \cite[corollary to Theorem~1]{neh}).
While all proofs of the \emph{existence part} of Theorem \ref{Th1} in the
literature are quite hard, the electrostatic interpretation allows us
to provide a simple and transparent proof of \emph{existence} and
\emph{uniqueness}.

The main argument for proving uniqueness is originally
due to Sarason and Suarez \cite{sasu} and has also been used  in \cite{gri}
and \cite{org}. The formulation of  \cite[Theorem 2]{org} can easily be
extended to cover our situation here. In our exposition the proof
is naturally based on the fundamental relation \eqref{eq15}.

We start with the problem of $n-1$ movable unit charges at positions
$t_1,\dots , t_{n-1}$ on the real line  and introduce their energy, which is
(neglecting some physically motivated factor)
\begin{equation}\label{eq23}
W(t_1, \dots, t_{n-1}): = \hspace*{-1ex} \sum_{1\le k,j\le n-1}
\hspace*{-1ex}\log \big|(t_k-\zeta_j)(t_k-\overline{\zeta}_j)\big|
-2\hspace*{-2ex}\sum_{1\le k<j\le n-1}\hspace*{-1ex} \log |t_j-t_k|.
\end{equation}
Since $W$ does not change upon permutations of its variables, we can
confine our considerations to the open subset $U$
of $\mathbb R^{n-1}$ where $t_1<\dots<t_{n-1}$.

\begin{lemma}\label{l6}
The function $W$ attains a global minimum at a point $(t_1,\dots,t_{n-1})$ in
$U$. This point is the only critical point of $W$ in $U$ and corresponds to
the unique solution of \eqref{eq17} with $t_1<t_2<\dots<t_{n-1}$.
\end{lemma}

\begin{proof}
1. A point $T:=(t_1, \dots, t_{n-1})$ is a critical point of $W$ if and only if
$({\partial W}/{\partial t_k})(T)=0$
for $k=1,\dots, n-1$. A straightforward computation shows that this is
equivalent to (\ref{eq17}).
\smallskip\par

2. In order to prove that $W$ attains a \emph{global minimum} in $U$,
we observe that the first sum in {\rm\eqref{eq23}} is
bounded from below, while the second sum tends to $+\infty$ whenever $T$
approaches a finite boundary point of $U$ (which implies that
$t_j-t_{j+1}\to 0$ for some $j$).
So it only remains to study the behavior of $W(T)$ when $\|T\|\to \infty$ in
some norm of $\mathbb R^{n-1}$. Let $C>0$ be a constant (depending
on $\zeta_1,\dots, \zeta_{n-1}$) such that
\[
|t-\zeta_j|\geq C(1+|t|),\qquad j=1,\dots, n-1, \quad t\in\mathbb R.
\]
Using  the inequality
\[
|t_j-t_k|\leq |t_j|+|t_k|\leq (1+|t_j|)(1+|t_k|)
\]
we estimate
\begin{align*}
W(t_1,&\dots, t_{n-1})\\
&\geq (n-1)\sum_{k=1}^{n-1}\log (C^2(1+|t_k|)^2)
-2\hspace*{-1ex}\sum_{1\le k<j\le n-1}\hspace*{-1ex}
\log \big((1+|t_j|)(1+|t_k|)\big)\\
&=C_1+2(n-1)\sum_{k=1}^{n-1}\log (1+|t_k|)-2(n-2)\sum_{k=1}^{n-1}\log(1+|t_k|)\\
&=C_1+2\sum_{k=1}^{n-1}\log(1+|t_k|)=C_1+2\log\prod_{k=1}^{n-1}(1+|t_k|)\\
&\geq C_1+2\log \left(1+\sum_{k=1}^{n-1}|t_k|\right)=C_1+2\log (1+\|T\|_1),
\end{align*}
and hence $W(T)\to\infty$ as $\|T\|_1\to \infty$. We conclude that $W$ attains
a global minimum at a point $T\in{\cal O}$ which represents a solution of
{\rm\eqref{eq17}}.
\smallskip\par

3. In order to show uniqueness (of the solution of \eqref{eq17} and hence
the  critical point of $W$ in $U$), we assume that $t_1< t_2<\dots<t_{n-1}$ and
$\tilde{t}_1< \tilde{t}_2<\dots<\tilde{t}_{n-1}$ are two solutions of
(\ref{eq17}). For $k=1,\dots, n-1$ let
\begin{alignat*}{2}
&S(x):=\prod_{j=1}^{n-1}(x-t_j), \quad
&&S_k(x):=\prod_{\substack{j=1\\j\neq k}}^{n-1}\frac{x-t_j}{t_k-t_j},\\
&\tilde{S}(x):=\prod_{k=j}^{n-1} (x-\tilde{t}_j),
&&\tilde{S}_k(x):=\prod_{\substack{j=1\\j\neq k}}^{n-1}
\frac{x-\tilde{t}_j}{\tilde{t}_k-\tilde{t}_j}.
\end{alignat*}
By  Lemma~\ref{l5}, both $S$ and $\tilde{S}$ are Stieltjes polynomials, i.e. they satify \eqref{eq16} for suitable Van Vleck polynomials. Going back from this equation in the  calculations in Section~\ref{sec1} we infer
they also satisfy equation (\ref{eq15}), and hence we have
\begin{equation}\label{eq21}
S(x)^2+\sum_{k=1}^{n-1}P(t_k)S_k(x)^2
=\tilde{S}(x)^2+\sum_{k=1}^{n-1}P(\tilde{t}_k)\tilde{S}_k(x)^2.
\end{equation}
The $2n-2$ polynomials $S_1^2,\dots, S_{n-1}^2, \tilde{S}_1^2, \dots ,
\tilde{S}_{n-1}^2$ are of degree $2n-4$, hence they must be linearly dependent,
i.e., there are real numbers $\alpha_k,\beta_k $, not all vanishing, such that
\[
\sum_{k=1}^{n-1}\alpha_k S_k(x)^2=\sum_{k=1}^{n-1}\beta_k\tilde{S}_k(x)^2.
\]
Multiplying this by $t\in \mathbb{R}$ and adding it to (\ref{eq21}) we get
\begin{equation}\label{eq22}
S(x)^2+\sum_{k=1}^{n-1}\big(P(t_k)+t\alpha_k\big)S_k(x)^2
=\tilde{S}(x)^2+\sum_{k=1}^{n-1}\big(P(\tilde{t}_k)+t\beta_k\big)\tilde{S}_k(x)^2
\end{equation}
for all $t\in\mathbb R$. Since $P(t_k)>0, P(\tilde{t}_k)>0$ we can
choose $t$ such that (at least) one of the numbers $P(t_k)+t\alpha_k$,
$P(\tilde{t}_k)+t\beta_k$ with $k\in \{1,\dots n-1\}$ vanishes while all
the others remain non-negative. Assume e.g. $P(t_{k_0})+t\alpha_{k_0}=0$ for some
$k_0\in\{1,\dots, n-1\}$. Since $S\big(t_{k_0}\big)=0$ and $S_k\big(t_{k_0}\big)=0$
for $k \not=k_0$, the left hand side of \eqref{eq22} vanishes at $x=t_{k_0}$.
But the right hand side of this equation cannot have other zeros than $\tilde{S}^2$,
hence $t_{k_0}=\tilde{t}_{j_0}$ for some $j_0\in\{1,\dots, n-1\}$.

If $n=2$ we are done. Otherwise the $2n-4$ polynomials of degree   $2n-6$
\[
\frac{S_k^2}{(x-t_{k_0})^2},\qquad \frac{\tilde S_j^2}{(x-\tilde{t}_{j_0})^2}\qquad
(1\leq k,\ j\leq n-1,\ k\neq k_0,\ j\neq j_0)
\]
are also linearly dependent, so that we have a non-trivial relation of the form
\[
\sum_{\substack{k=1\\k\neq k_0}}^{n-1}\gamma_k S_k(x)^2
=\sum_{\substack{j=1\\j\neq j_0}}^{n-1}\delta_j\tilde{S}_j(x)^2.
\]
With (\ref{eq21}) we have for $t\in\mathbb R$
\begin{alignat*}{2}
S(x)^2 &+P(t_{k_0})S_{k_0}(x)^2 + \sum_{\substack{k=1\\k\neq k_0}}^{n-1}
(P(t_k)+t\gamma_k) S_k(x)^2 \\
&=\tilde{S}(x)^2+P(\tilde{t}_{j_0})\tilde{S}_{j_0}(x)^2
+\sum_{\substack{j=1\\j\neq j_0}}^{n-1}
(P(\tilde{t}_j)+t\delta_j)\tilde{S}_j(x)^2.
\end{alignat*}
As above we find now  another pair of equal zeros of $S$ and $\tilde{S}$.
Proceeding inductively we finally get that both solutions of (\ref{eq17})
are identical. \qed
\end{proof}

The solution of \eqref{eq10} can now be reduced to \eqref{eq17}
using an appropriate M\"{o}bius transformation. We summarize the results
for both equations in the following theorem.

\begin{theorem}
\label{Th2}
For arbitrary $\zeta_1,\dots, \zeta_{n-1}\in \mathbb H$ we have:
\begin{enumerate}
\item[{\rm (i)}]
The equation {\rm (\ref{eq17})} has a unique solution with
$t_1<t_2<\dots <t_{n-1}$. It realizes the global minimum of the
potential energy $W$ defined in {\rm (\ref{eq23})}, and is (up to
a rearrangement of the $t_k$) the only critical point of $W$.
\item[{\rm (ii)}]
Let additionally  $t_0:=-\infty$, $t_n:=+\infty$ and fix $k_0\in
\{1,2,\dots, n\}$. For every $x_{k_0}\in (t_{k_0-1}, t_{k_0})$
there are  unique points $x_k\in (t_{k-1}, t_k)$ for
$k=1,\dots, k_0-1, k_0+1, \dots,n$ such that  equation {\rm
(\ref{eq10})} is satisfied. All $x_k$ depend continuously and
monotonously on $x_{k_0}$, and we have $x_k \to t_{k-1}$ if
$x_{k_0}\to t_{k_0-1}$ and  $x_k \to t_k$ if $x_{k_0}\to
t_{k_0}$.
\end{enumerate}
\end{theorem}
\begin{proof}
Assertion~(i) has been proved in Lemma~\ref{l6}.
In order to show the \emph{existence part} of (ii) let $a>0$, $b\in\mathbb R$
be arbitrary and define $s_1,\dots, s_{n-1} $  by (\ref{eq4}).
According to Lemma~\ref{l4} and Lemma~\ref{l5}, the function $g$
from (\ref{eq11}) has critical points $\zeta_1,\dots, \zeta_{n-1}$.
Since $g(x)\to\pm\infty $ as $x\to t_k\mp 0$ ($k=0,\dots,n)$ and $g$  is strictly
increasing in each interval $(t_{k-1},t_k)$ ($k=1,\dots,n)$,
any such interval contains a unique solution $x_k$ of $g(x_k)=g(x_{k_0})$.
It is clear that these points have the claimed properties concerning
continuity, monotonicity and limit behavior. The function $f$ defined by
\begin{equation}\label{eq31}
f(x):=-\frac{1}{g(x)-g(x_{k_0})}
\end{equation}
is rational and can be represented as a quotient $f=p/q$, where $p$ is a real
polynomial of degree $n-1$, and $q$ is a real polynomial of degree $n$,
respectively. Since $f$ has simple poles in $x_1,\dots, x_n$,
the partial fraction decomposition of $f$ has the form
(\ref{eq1}). From the monotonicity of $g$ we obtain that $f$ is
increasing between two poles and thus $r_k>0$. Because $f$ and $g$ have
the same critical points $\zeta_1,\dots,\zeta_{n-1}$, we can combine
Lemma~\ref{l3} and Lemma~\ref{l5} to obtain (\ref{eq10}).

In order to show the \emph{uniqueness assertion} of (ii) we let
$\tilde{x}_1<\dots < \tilde{x}_n$ be any solution of (\ref{eq10})
with $\tilde{x}_{k_0}=x_{k_0}$ and define $\tilde{r}_k>0$ by
(\ref{eq7}) for some $c>0$ and $x_k$ replaced by $\tilde{x}_k$. Let 
 the function $\tilde{f}$ be of the form (\ref{eq1}) with poles at $\tilde{x}_k$ and residues
$\tilde{r}_k$. According to Lemma~\ref{l3} and Lemma~\ref{l5}, $\tilde{f}$ has
critical points $\zeta_1,\dots, \zeta_{n-1}$. Since $\tilde{f}$
is real-valued and increasing between any two of its poles, it has zeros
$\tilde{t}_k\in (\tilde{x}_k, \tilde{x}_{k+1})$ for $k=1,\dots, n-1$.
Hence $\tilde{g}:=-1/\tilde{f}$ has the form
\[
\tilde{g}(x)=\tilde{a}x+\tilde{b}-\sum_{k=1}^{n-1}
\frac{\tilde{s}_k}{x-\tilde{t}_k}, \qquad \tilde{a}
:=\bigg(\sum_{k=1}^{n-1}\tilde{r}_k\bigg)^{-1}>0,\quad
\tilde{s}_k>0,  \quad \tilde{b}\in\mathbb R.
\]
The critical points of $\tilde{g}$ are the same as those of
$\tilde{f}$, hence we can again invoke Lemma~\ref{l4} and Lemma~\ref{l5}
to conclude that $\tilde{t}_1,\dots, \tilde{t}_{n-1}$ solve (\ref{eq17}).
From the uniqueness stated in (i) it follows that $\tilde{t}_k=t_k$ for
$k=1,\dots, n-1$. Consequently, by {\rm\eqref{eq30}} and {\rm\eqref{eq14}},
\[
g'(x)= \frac{a P(x)}{S(x)}, \qquad \tilde{g}'(x)= \frac{\tilde{a}P(x)}{S(x)}.
\]
Hence $g'=d\tilde{g}'$ for some $d>0$, and therefore $g=d\tilde{g}+e$ with
$e\in\mathbb R$. Using (\ref{eq31}) we obtain
\[
-\frac{1}{f(x)}+g(x_{k_0})=g(x)=d\tilde{g}(x)+e=-\frac{d}{\tilde{f}(x)}+e
\]
and since $f$ and $\tilde{f}$ both have poles at $x_{k_0}$ it
follows that $e=g(x_{k_0})$. Thus $\tilde{f}=df$, in particular
$\tilde{x}_k=x_k$ for $k=1,\dots, n$. \qed
\end{proof}

The proof of Theorem~\ref{Th2} also indicates how  to construct
rational functions $f$ and $g$ of the form \eqref{eq1} or \eqref{eq11} with given
critical points $\zeta_k$  from the solutions of \eqref{eq10} and \eqref{eq17},
respectively. Setting $B:=T^{-1} \circ f\circ T$ we obtain by Lemma~\ref{l1} a Blaschke product with critical points $\xi_k=T^{-1}(\zeta_k)$, and the same is true for    $B:=T^{-1} \circ g\circ T$ by Lemma~\ref{l2}. This shows again the existence of a Blaschke product of degree $n$ with $n-1$ given critical points in $\mathbb D$.  In the same way the uniqueness statement of Theorem~\ref{Th1} can be inferred from the uniqueness statements in Theorem~\ref{Th1} and we have thus provided an independent proof  of this result.
\medskip

{\bfseries Example 2.} Consider the case where $ \zeta_1=\dots= \zeta_{n-1}=\mathrm{i}$.
Since $T^{-1}(\mathrm{i})=0$, we have $\xi_1=\dots=\xi_{n-1}=0\in\mathbb  D$,
i.e., this problem is equivalent to  finding a Blaschke product $B$
of degree $n$ with  critical point  of order $n-1$ at $z=0$. An
obvious solution with $B(1)=1$ is $B(z)=z^n$, and
the solutions of $B(z)=1$ are the $n$th roots of unity
$\eta_k=\mathrm{e}^{2\pi \mathrm{i}\,k/n}$ ($k=0,1,\dots, n-1$).
The unique solution of (\ref{eq17}) is then given by the images of
$\eta_1,\dots,\eta_{n-1}$ under the mapping $T$,
\[
t_k=T(\eta_k)=\mathrm{i}\,
\frac{1+\mathrm{e}^{2\pi \mathrm{i}\,k/n}}{1-\mathrm{e}^{2\pi \mathrm{i}\,k/n}}
\]
The case where all $ \zeta_1,\dots, \zeta_{n-1}$ coincide at another
point of $\mathbb H$ can be reduced to this cases by an appropriate automorphism
$z\mapsto cz+d$ with $c>0, d\in\mathbb R$.
\medskip

By now we know that the Stieltjes polynomials (\ref{eq12})
(corresponding to the one-parameter family of solutions of
(\ref{eq10})) and (\ref{eq14}) (corresponding to the solution of
(\ref{eq17})) are associated with  certain Van Vleck polynomials
satisfying (\ref{eq8}) and (\ref{eq16}), respectively.
It turns out that all these polynomials are the same, so
that we can speak  about the \emph{Van Vleck polynomial for the
critical points $\zeta_k$}.

\begin{lemma}\label{l9}
Let $\zeta_1,\dots,\zeta_{n-1}\in\mathbb H$ and let
$t_1,\dots,t_{n-1}$ and $x_1,\dots,x_n$ be the solutions described
in Theorem \ref{Th2}. Then the Van Vleck polynomials $R$ and
$\tilde{R}$, corresponding to  solutions {\rm (\ref{eq12})} and
{\rm (\ref{eq14})} of {\rm (\ref{eq8})} and {\rm (\ref{eq16})},
respectively, coincide.
\end{lemma}

\begin{proof}
Let $f$ be defined by (\ref{eq1}) and (\ref{eq7}) with $c>0$ and let
$g$ be given by (\ref{eq11}) and (\ref{eq4}) with $a>0$ and
$b\in\mathbb R$. As we have seen in the proof of Theorem \ref{Th2},
$f$ and $g$ are connected by the equation
\[
-\frac{d}{f(x)}+e=g(x)
\]
with constants $d>0$ and $e\in\mathbb R$ (i.e. they differ by an automorphism of
$\mathbb H$ that maps $0$ to $\infty$). Differentiating this equation we obtain
\[
\frac{d\,f'(x)}{f(x)^2}=g'(x),
\]
and plugging in (\ref{eq2}) and (\ref{eq30}) we arrive at
\[
\frac{cd\,P(x)}{Q(x)^2f(x)^2}=\frac{a\,P(x)}{S(x)^2},
\]
so that $a\,Q(x)^2f(x)^2=cd\,S(x)^2$. Since $Q(x)\sim x^{n}$,
$S(x)\sim x^{n-1}$, $f(x)\sim -c\,x^{-1}$ as $x \to \infty$,
we have shown the remarkable identity
\begin{equation}\label{eq32}
-c\,S(x)=Q(x)f(x).
\end{equation}
Differentiating (\ref{eq32}) we get
\begin{equation} \label{eqS}
-c\,S'=Q'f+Qf',\qquad -c\,S''=Q''f+2Q'f'+Qf''.
\end{equation}
Recall from (\ref{eq2}) that $f'=c\,P/Q^2$,
and thus
\begin{equation} \label{eqPQ}
PQf'' =\frac{cP}{Q^2}(P'Q-2PQ')=f'(P'Q-2PQ').
\end{equation}
Using (in this order) {\rm\eqref{eq16}}, {\rm\eqref{eqS}}, {\rm\eqref{eq8}},
{\rm\eqref{eqPQ}}, and {\rm\eqref{eq32}}, we get
\begin{align*}
-c\tilde{R}S&=-c(P'S'-PS'')
=P'(Q'f+Qf')-P(Q''f+2Q'f'+Qf'')\\
&=(P'Q'-PQ'')f+(P'Q-2PQ')f'-PQf''
=RQf=-cRS,
\end{align*}
so that $R=\tilde{R}$. \qed
\end{proof}
As a consequence of this lemma we obtain that for the Van Vleck
polynomial $R$ the solution space of the Lam\'{e} equation
\begin{equation} \label{eq.Lame}
PY''-P'Y'+RY = 0
\end{equation}
consists only of polynomials, and that its general solution is given by
\begin{equation}\label{eq33}
Y(x)=\lambda Q(x)+\mu S(x), \qquad \lambda,\mu\in\mathbb{R}.
\end{equation}
Fixing $\lambda \neq 0$ and letting $\mu$ run through $\mathbb R$,
the zeros of the polynomials (\ref{eq33}) run through all
equilibrium positions of $n$ points. Solutions with $\lambda=0$
correspond to the limit configuration where one charge escaped to
infinity and only $n-1$ charges remain on the real line.
The preceding result also confirms what was said in \cite[Remark 3]{org}
about possible polynomial solutions of the Lam\'{e} equation.

Once two Stieltjes polynomials $Q$ and $S$ as solutions of {\rm\eqref{eq.Lame}}
are known, the polynomial $P$ and the Van Vleck polynomial $R$ can be
reconstructed as shows the following result.

\begin{lemma}\label{l10}
Let $Q$ and $S$ be monic polynomial solutions of degree $n$ and $n-1$,
respectively, of the Lam\'{e} equation \eqref{eq.Lame}. Then we have
\begin{equation}\label{eqPQS}
P(x)=S(x)Q'(x)-S'(x)Q(x)
\end{equation}
and
\begin{equation}\label{eqRQS}
R(x)=S'(x)Q''(x)-S''(x)Q'(x).
\end{equation}
\end{lemma}
\begin{proof}
Since  $Q$ and $S$ are linearly independent they form a fundamental system
for the differential equation. Their Wronskian
\begin{equation*}
w(x):=\begin{vmatrix}
S(x)&Q(x)\\S'(x)& Q'(x)
\end{vmatrix}
=S(x)Q'(x)-S'(x)Q(x)
\end{equation*}
is therefore non-vanishing and  Abel's formula applied to \eqref{eq8}
implies that for $x,x_0\in\mathbb R$
\begin{equation*}
 w(x)=w(x_0)\exp\left(-\int_{x_0}^x\frac{-P'(t)}{P(t)}\,dt\right)
=\frac{w(x_0)}{P(x_0)}P(x).
\end{equation*}
Since $w$ and $P$ are easily seen to be monic, the two
polynomials coincide and we have proved  \eqref{eqPQS}.
Differentiating this equation we get
\[
P'(x)=Q''(x)S(x)-Q(x)S''(x),
\]
and using {\rm\eqref{eq.Lame}}, {\rm\eqref{eqPQS}}, we finally arrive at
\[
RQ = P'Q'- PQ''=(Q''S-QS'')Q'-(SQ'-S'Q)Q'' = (S'Q''-S''Q')Q,
\]
and \eqref{eqRQS} follows. \qed
\end{proof}


\section{Convex cones and moment problems} \label{convex}

In this section we will show that a well-known problem in
\emph{moment theory} is also equivalent to the determination of a
Blaschke product with given critical points.
We start with some notions and facts from convex analysis,
our main sources are \cite{akr}, \cite{bar}, \cite{bova}, \cite{ioh},
\cite{kast}, \cite{krnu}.

A \emph{convex cone} $\mathcal{C}$ is a subset of a real vector space,
such that $u,v\in \mathcal{C}$ implies $\lambda u+\mu v\in \mathcal{C}$
for all positive $\lambda$ and $\mu$. One standard example is
the convex cone of \emph{real non-negative} polynomials of degree
at most $2n-2$,
\[
{\cal P}_{2n-1}:=\left\{(p_0,\dots, p_{2n-2})\in \mathbb R^{2n-1}:
p_0+p_1x+\dots + p_{2n-2}x^{2n-2}\geq 0,\ \text{on}\ \mathbb R\right\},
\]
where we identified a polynomial with the vector of its coefficients.
The interior of this convex cone consists of all  positive
polynomials; hence the polynomial $P$ from (\ref{eq13}) belongs  to
the interior of ${\cal P}_{2n-1}$ provided that $\zeta_k\in\mathbb H $
for $k=1,\dots, n-1$.
Recall that $P$ encodes the given data, i.e., the critical points
of the Blaschke product.

Another important example is the convex cone of \emph{symmetric positive
semi-definite matrices},
\[
\mathcal{S}_+^{n}:=\{M\in \mathbb R^{n\times n}:
M=M^\top \text{ and } M\succeq 0\}.
\]
The interior of $\mathcal{S}_+^{n}$ is the convex cone of
\emph{symmetric positive definite matrices},
\[
\mathcal{S}_{++}^{n}:=\{M\in \mathbb R^{n\times n}:
M=M^\top \text{ and } M\succ 0\}.
\]
Furthermore, we consider the \emph{moment cone}
\[
{\cal M}_{2n-1}:=\left \{c=(c_0,\dots, c_{2n-2})\in \mathbb R^{2n-1}:
c_k=\int_{-\infty}^\infty t^k\,d\sigma,\
\sigma\in M_{2n-1}\right\},
\]
where  $M_{2n-1}$ is the set of nonnegative measures $\sigma$ on $\mathbb R$
such that
\[
\int_{-\infty}^\infty |t|^k\,d\sigma<\infty, \quad k=0,\dots, 2n-2.
\]
The set ${\cal M}_{2n-1}$ is the \emph{conic hull} of the \emph{moment curve}
\begin{equation}\label{equ}
C_{2n-1}:=\left\{u(t)=(1,t,\dots, t^{2n-2})\in \mathbb{R}^{2n-1}:
t\in\mathbb R\right\},
\end{equation}
i.e.,  ${\cal M}_{2n-1}$ is the smallest convex cone containing $C_{2n-1}$, cf. Theorem 2.1 in  chapter V of \cite{kast}. 
Note that ${\cal M}_{2n-1}$ is not a closed subset of $\mathbb
R^{2n-1}$ since points in the closure of this set can involve
representations with ``mass at infinity''. More precisely, $c$ belongs to
the closure $\overline{\cal M}_{2n-1}$ of ${\cal M}_{2n-1}$ if and only if
it has the representation
\begin{equation} \label{eq24}
c_k=\int_{-\infty}^\infty t^k\,d\sigma\quad (k=0,\dots, 2n-3), \qquad
c_{2n-2}=\int_{-\infty}^\infty t^{2n-2}\,d\sigma+\lambda,
\end{equation}
with $\sigma\in M_{2n-1}$ and non-negative $\lambda$ (representing a mass
at infinity). An alternative characterization of $\overline{\cal M}_{2n-1}$
is
\[
\overline{\cal M}_{2n-1}=\left\{c\in\mathbb R^{2n-1}: H(c)\succeq
0\right\},
\]
where
\begin{equation} \label{eq.hankel}
H(c)=H(c_0,\dots c_{2n-2}):=
\begin{bmatrix}
c_0&c_1&\cdots &c_{n-1}\\
c_1&c_2&\cdots &c_{n}\\
\vdots&\vdots &&\vdots\\
c_{n-1}&c_{n}&\cdots &c_{2n-2}
\end{bmatrix}
\end{equation}
denotes the \emph{Hankel matrix} associated with $c$. Also, a point
$c\in\mathbb R^{2n-1}$ is an inner point of ${\cal M}_{2n-1}$ if and
only if $H(c)$ is positive definite. Using the linear mapping
\begin{equation} \label{eq.hankelmap}
H:\mathbb{R}^{2n-1} \rightarrow \mathbb{R}^{n\times n},\ c \mapsto H(c),
\end{equation}
we can therefore write
\[
\overline{{\cal M}}_{2n-1}=H^{-1}\big(\mathcal{S}_+^{n}\big), \qquad
\interior {\cal M}_{2n-1}=H^{-1}\big(\mathcal{S}_{++}^{n}\big).
\]
A representation (\ref{eq24}) of a point  $c$ in
(the closure of) the moment cone is  usually not unique. Therefore
one searches for \emph{canonical representations} of $c$ where the
measure $\sigma\in M_{2n-1}$ is concentrated at a finite number of
points. For example, if $\sigma $ is concentrated at points
$t_1,t_2,\dots,t_{n-1}$ with masses $\sigma_1,\dots,
\sigma_{n-1}>0$ and mass $\lambda>0$ at infinity  we have
\begin{equation} \label{eq26}
c_k=\sum_{j=1}^{n-1}\sigma_j t_j^k,\quad (k=0,\dots, 2n-3), \qquad
c_{2n-2}=\sum_{j=1}^{n-1}\sigma_j t_j^{2n-2}+\lambda.
\end{equation}
If $\sigma$ is concentrated at points $x_1,x_2,\dots,x_n$ with
masses $\varrho_1,\dots, \varrho_n>0$ and has no mass at infinity
then we have
\begin{equation}\label{eq28}
c_k=\sum_{j=1}^n\varrho_j x_j^k,\qquad k=0,\dots, 2n-2.
\end{equation}
For given moments $c_k$, the \emph{moment problem} consists
in finding the \emph{roots} $t_k, x_k$ and the corresponding \emph{weights}
$\sigma_k$, $\varrho_k$, featuring in the representations {\rm\eqref{eq26}},
{\rm\eqref{eq28}}, respectively.
The following theorem
is well known (see e.g. \cite{akr}, \cite{kast}).
\begin{theorem}\label{Th3}
For any $c\in \interior {\cal M}_{2n-1}$ the following assertions hold:
\begin{enumerate}
\item[{\rm (i)}]
There is a representation {\rm (\ref{eq26})}
with uniquely determined roots $t_1<t_2<\dots <t_{n-1}$,
weights $\sigma_1,\dots, \sigma_{n-1}>0$ and $\lambda>0$.
\item[{\rm (ii)}]
Let additionally  $t_0:=-\infty$, $t_n:=+\infty$, and fix $k_0\in
\{1,2,\dots, n\}$. Then for every $x_{k_0}\in (t_{k_0-1}, t_{k_0})$
there are unique roots $x_k\in (t_{k-1}, t_k)$ for
$k=1,\dots, k_0-1, k_0+1, \dots,n$ and weights
$\varrho_1,\dots, \varrho_n>0$ such that  equation {\rm
(\ref{eq28})} is satisfied. All $x_k$ depend continuously and
monotonically on $x_{k_0}$.
\end{enumerate}
\end{theorem}

\noindent
Before going on, we observe that the moment problem
{\rm\eqref{eq28}} can be rephrased as a matrix factorization problem. Let
\begin{equation*}
V(x_1,\dots, x_n):=
\begin{bmatrix}
1&1&\cdots&1\\
x_1&x_2&\cdots&x_n\\
\vdots&\vdots&&\vdots\\
x_1^{n-1}&x_2^{n-1}&\cdots &x_n^{n-1}
\end{bmatrix}
\end{equation*}
denote the \emph{Vandermonde matrix} of the points $x_1,\dots, x_n$. With the
abbreviations $V:=V(x_1,\dots, x_n)$, $D:=\diag (\varrho_1,\dots, \varrho_n)$,
and $H(c)$ from {\rm\eqref{eq.hankel}}, the equations (\ref{eq28}) can be
rewritten as
\begin{equation}
\label{eq29}
H(c)=VDV^\top.
\end{equation}
This representation is known as the \emph{Vandermonde factorization} of
the (positive definite) Hankel matrix $H(c)$ (see Heinig and Rost \cite{hero2},
\cite{hero1}).

There is also a nice way to obtain the roots
$t_k, x_k$. The polynomial $D_{n-1}$ defined by
\begin{equation} \label{eq.detD}
D_{n-1}(x):=\det
\begin{bmatrix}
c_0&c_1&\cdots &c_{n-2}&1\\
c_1&c_2&\cdots &c_{n-1}&x\\
\vdots&\vdots &&\vdots\\
c_{n-1}&c_{n}&\cdots &c_{2n-3}&x^{n-1}
\end{bmatrix}
\end{equation}
is of (exact) degree $n-1$ and its roots  are  $t_1,\dots,
t_{n-1}$. To see this we observe that in view of $\det H(c_0,\dots,
c_{2n-4})>0$ the first $n-1$ columns of the matrix in
{\rm\eqref{eq.detD}} are linearly
independent and in view of (\ref{eq26}) each of them is a linear
combination of the first $n-1$ columns of $V(t_1,\dots,t_{n-1},x)$.
Hence the spaces spanned by the first $n-1$ columns of these two matrices
coincide, such that  the two determinants $D_{n-1}(x)$ and $\mathrm{det}\,
V(t_1,\dots,t_{n-1},x)$ vanish for the same values of $x$.

Similarly, in view of $D_{n-1}(x_{k_0})\neq 0$ the first $n$ columns of the matrix in
\[
E_n(x):=\det \begin{bmatrix}
c_0& c_1&\cdots &c_{n-2}&1&1\\
c_1&c_2&\cdots &c_{n-1}&x_{k_0}&x\\
\vdots&\vdots&&\vdots&\vdots&\vdots\\
c_{n-1}&c_n&\cdots &c_{2n-3}&x_{k_0}^{n-1}&x^{n-1}\\
c_n&c_{n+1}&\cdots & c_{2n-2}&x_{k_0}^n&x^n
\end{bmatrix}
\]
are linearly independent and by (\ref{eq28}) we find that the zeros
of $E_n(x)$ coincide with the zeros of
$\det V(x_1,\dots, x_n,x)$. Hence the zeros of
$E_n(x)$ are the roots described in
Theorem~\ref{Th3} (ii).
\smallskip

There is an evident similarity between Theorem~\ref{Th2} and Theorem~\ref{Th3}.
As we will see in a moment,
the solutions described in both theorems coincide if we map the polynomial
$P\in \interior {\cal P}_{2n-1}$ from (\ref{eq13}) to an
appropriate point $c\in\interior{\cal M}_{2n-1}$. This mapping was
investigated by Nesterov~\cite{nest} and will be described below.
Hachez and  Nesterov~\cite{hane} used it together with the
Vandermonde factorization (\ref{eq29}) to represent a positive
polynomial as a weighted sum of squares of Lagrange interpolation
polynomials as in  formula (\ref{eq5}). This yields
another equivalent reformulation of the original Blaschke product
problem.

In order to describe a mapping between $\interior {\cal P}_{2n-1}$
and $\interior{\cal M}_{2n-1}$ we have to recall the notion of
duality. If $K$ is a convex cone in a real Hilbert space $V$ with
scalar product $\langle\cdot,\cdot\rangle$, its \emph{dual cone} is
$K^*:=\{x\in V:\langle x,y\rangle \geq 0 \text{ for all } y\in K\}$.
In the sequel we use scalar products for vectors
$x=(x_k), y=(y_k) \in \mathbb R^p$ and matrices
$X=(x_{kj}),Y=(y_{kj})\in\mathbb R^{p\times q}$ defined by
\[
\langle x,y\rangle:=\sum_{k=1}^p x_k y_k,  \qquad
\langle X,Y\rangle:=\sum_{k=1}^p\sum_{j=1}^q x_{kj}y_{kj}.
\]
With respect to these scalar products we have the following duality relations:
\[
{\cal P}_{2n-1}^*=\overline{\cal M}_{2n-1}, \qquad
{\cal M}_{2n-1}^*={\cal P}_{2n-1},\qquad (\mathcal{S}_+^n)^*=\mathcal{S}_+^n.
\]
Let
\[
H^*: \mathbb{R}^{n\times n} \mapsto \mathbb R^{2n-1},\ X=(x_{ij}) \mapsto y=(y_k),
\quad \text{with } y_k=\sum_{\substack{i,j=0\\i+j=k}}^{n-1} x_{ij},
\]
be the mapping which sends a  matrix $X$ to the vector $y$ of its \emph{anti-diagonal sums}.
The mapping $H^*$ is dual to the mapping $H$ defined in \eqref{eq.hankelmap}  in the sense
that
\[
\langle H(x),X\rangle=\langle x, H^*(X)\rangle, \qquad
x\in \mathbb R^{2n-1},\ X\in \mathbb R^{n\times n}.
\]
Moreover, the cone of non-negative (positive) polynomials can be obtained as
the image of the cone of non-negative (positive) definite matrices,
\[
{\cal P}_{2n-1}=H^*\big(\mathcal{S}^n_+\big),\qquad
\interior {\cal P}_{2n-1}=H^*\big(\mathcal{S}_{++}^n\big).
\]
The mapping $H^*$ has an interesting interpretation when we identify a
vector $c\in\mathbb R^{2n-1}$ with the polynomial
$c(x)=\langle c, u(x)\rangle=c_0+c_1x+\dots+ c_{2n-2}x^{2n-2}$,
and a matrix $X\in\mathbb R^{n\times n}$ with the polynomial in two variables
\[
X(x,y):=\big\langle u(x), u(y)X^\top\big\rangle
=\sum_{i=0}^{n-1}\sum_{j=0}^{n-1}x_{ij}x^iy^j,
\]
where $u$ is the function defined in \eqref{equ}. Then $H^*$ is the operator
of equating variables that maps $X=X(x,y)$ to the polynomial
\begin{equation}
\label{eq34}
(H^*X)(x)=X(x,x).
\end{equation}
Nesterov~\cite{nest} introduced the mapping
\begin{equation*}
c\mapsto N(c):=H^*\left(H(c)^{-1}\right).
\end{equation*}
which is defined for $c\in\mathbb R^{2n-1}$ whenever $H(c)$ is nonsingular.
If $H(c)$ is even positive definite, the inverse matrix $H(c)^{-1}$ is also
positive definite and we have for $x\in\mathbb R$
\[
N(c)(x)=H^*\left(H(c)^{-1}\right)(x)=
\left\langle u(x),u(x) H(c)^{-\top}\right\rangle >0,
\]
i.e., $N(c)$ is (the coefficient vector of) a positive polynomial. Thus
the \emph{Nesterov mapping} $N$  maps $\interior {\cal M}_{2n-1}$ to
$\interior{\cal P}_{2n-1}$.  Nesterov~\cite{nest} also remarked that $N$ is
the (negative) gradient of the function
\[
h(c):=-\log\det H(c), \qquad c\in \interior{\cal M}_{2n-1},
\]
which is a  \emph{(strongly non-degenerate self-concordant) barrier functional}
for ${\cal M}_{2n-1}$, see \cite{nene} or \cite{ren} for definitions and basic
properties. Hence the gradient of $h$ is a \emph{bijection} of
$\interior {\cal M}_{2n-1}$ onto the interior $\interior {\cal P}_{2n-1}$ of
its dual cone ${\cal P}_{2n-1}$.
We will give an alternative proof of this fact, demonstrating
how the Nesterov mapping $N$ connects Theorem~\ref{Th2} and Theorem~\ref{Th3}.

\begin{theorem}\label{thm4}
Let $p\in \interior {\cal P}_{2n-1}$ be a positive polynomial with  zeros
$\zeta_1,\dots,\zeta_{n-1}$  in the upper half plane.
Then there is a unique  $c\in \interior {\cal M}_{2n-1}$ such that  $p=N(c)$.
If   $x_1,\dots,x_n$ and $t_1,\dots,t_{n-1}$ satisfy \eqref{eq10} and
\eqref{eq17}, respectively, then the canonical representations \eqref{eq26} and
\eqref{eq28} hold with the  positive numbers
\begin{equation} \label{eq.sigrho}
\sigma_k=\frac{1}{p(t_k)}\ (k=1,\dots, n-1),\
\varrho_k= \frac{1}{p(x_k)} \ (k=1,\dots, n),\
\lambda=1/p_{2n-2},
\end{equation}
where $p_{2n-2}$ is the leading coefficient of $p$.
Conversely, if \eqref{eq26} and \eqref{eq28} are satisfied for $x_1,\dots,x_n$,
$t_1,\dots,t_{n-1}$ and positive numbers $\sigma_k,\varrho_k,\lambda$, then
also the equilibrium conditions \eqref{eq10} and \eqref{eq17} are true.
Moreover, $\sigma_k$, $\varrho_k$, and $\lambda$ satisfy {\rm\eqref{eq.sigrho}}.
\end{theorem}

The proof of Theorem~\ref{thm4} is split into several lemmas.
Without loss of generality we always assume that $x_1<\dots< x_n$
and $t_1<\dots<t_{n-1}$.

In the following we denote by $\bez (v,w)$ the \emph{Bezoutian} of
$v,w\in \mathbb R^{n+1}$ (interpretable as polynomial of degree $\leq n$),
which is the matrix $B=(b_{ij})\in\mathbb R^{n\times n}$ defined by
\[
B(x,y)=\sum_{i=0}^{n-1}\sum_{j=0}^{n-1}b_{ij}x^iy^j=\frac{v(x)w(y)-w(x)v(y)}{x-y},
\]
see~\cite{fuh}, \cite{hero2}, \cite{hero1} for more information on this topic.
Recall that the inverse of a non-singular Hankel matrix is a non-singular
Bezoutian, and vice versa.

\begin{lemma}
For each $p\in \interior{\cal P}_{2n-1}$ there exists a vector $c\in\mathbb R^{2n-1}$
with $N(c)=p$.
\end{lemma}

\begin{proof}
Let $p(x)=p_{2n-2}x^{2n-2}+\dots+ p_0$ be a  positive polynomial with zeros
$\zeta_k\in\mathbb H$ and $\overline{\zeta}_k$ for  $k=1,\dots, n-1$.
Then $p(x)=p_{2n-2}P(x)$ with $p_{2n-2}>0$ and the monic polynomial $P$ from
\eqref{eq13}. Let $Q$ and $S$ be the polynomials \eqref{eq12} and \eqref{eq14}
with zeros at the equilibrium points $x_1<\dots< x_n$ and $t_1<\dots<t_{n-1}$
according to Theorem \ref{Th2}.
Let further  $B=(b_{ij})=\bez (Q,S)$ be the Bezoutian of $Q$ and $S$, i.e.,
\begin{equation}\label{eq37}
B(x,y)=\sum_{i=0}^{n-1}\sum_{j=0}^{n-1} b_{ij}x^iy^j
=\frac{Q(x)S(y)-S(x)Q(y)}{x-y}.
\end{equation}
The application of $H^*$ yields in view of \eqref{eq34}
\[
(H^*B)(x)=B(x,x)=\lim_{y\to x}\frac{Q(x)S(y)-S(x)Q(y)}{x-y}
= -Q(x)S'(x)+S(x)Q'(x).
\]
We have shown in equation \eqref{eqPQS} of Lemma \ref{l10} that $P$ is the Wronskian
of $S$ and $Q$, from which we conclude that $H^*B=P$. Since $B$ is a non-singular
Bezoutian, its inverse is a Hankel matrix, hence there is $d\in\mathbb R^{2n-1}$
with $H(d)=B^{-1}$. This vector $d$ satisfies $N(d)=H^*\left(H(d)^{-1}\right)=H^*(B)=P$.
Defining $c:=p_{2n-2}^{-1}d$ we get $N(c)=p_{2n-2}N(d)=p_{2n-2}P=p$. \qed
\end{proof}

We will soon see that even $c\in \interior {\cal M}_{2n-1}$. Before that we show:

\begin{lemma}
If $x_1,\dots, x_n$ satisfy {\rm\eqref{eq10}} and $\varrho_1,\dots\varrho_n$ are
given by {\rm\eqref{eq.sigrho}}, then the canonical representation \eqref{eq28} holds.
\end{lemma}

\begin{proof}
We show that for $k=0,\dots, 2n-2$
\[
e_k:=\sum_{j=1}^n\varrho_j x_j^k 
\]
is equal to $c_k$. In view of \eqref{eqPQS} and $Q(x_k)=0$
we have  $Q'(x_k)S(x_k)=P(x_k)$ and hence
\[
\varrho_k =\frac{1}{p(x_k)}=\frac{1}{p_{2n-2}Q'(x_k)S(x_k)},\qquad k=1,\dots, n.
\]
Putting $y=x_k$ in \eqref{eq37} we find
\[
B(x,x_k)=\sum_{i=0}^{n-1}\sum_{j=0}^{n-1}b_{ij}x^ix_k^j
=\frac{Q(x)S(x_k)-S(x)Q(x_k)}{x-x_k}=Q_k(x)Q'(x_k)S(x_k),
\]
where $Q_k$ are the Lagrange interpolation polynomials defined in \eqref{lagrange}.
Letting $x=x_l$ we obtain
\[
\sum_{i=0}^{n-1}x_l^i p_{2n-2}\varrho_k\sum_{j=0}^{n-1}b_{ij}x_k^j
=Q_k(x_l)=\delta_{kl},\qquad k,l=1,\dots, n,
\]
i.e., the matrix $Y$ with entries
\[
y_{ik}:=p_{2n-2}\varrho_k\sum_{j=0}^{n-1}b_{ij}x_k^j,\qquad
i=0,\dots, n-1, \;k=1,\dots, n,
\]
is a right inverse of the transposed Vandermonde matrix $V=V(x_1,\dots, x_n)^\top$.
Since $V$ is a square matrix, $Y$ is also left inverse to $V$, so that for
$k,l=0,\dots, n-1$
\begin{alignat*}{2}
\delta_{kl}&=\sum_{i=1}^{n}y_{ki}x_i^l=
\sum_{i=1}^{n}p_{2n-2}\varrho_i\sum_{j=0}^{n-1}b_{kj}x_i^jx_i^l\\
&=p_{2n-2}\sum_{j=0}^{n-1}b_{kj}\sum_{i=1}^{n}\varrho_ix_i^{j+l}
=p_{2n-2}\sum_{j=0}^{n-1}b_{kj}e_{j+l}.
\end{alignat*}
So the Hankel matrix $H(e)$ is inverse to the
Bezoutian $p_{2n-2}B=H(c)^{-1}$ and thus $e_k=c_k$ for
$k=0,\dots, 2n-2$.  \qed
\end{proof}

As a consequence of this lemma and \eqref{eq29} we get that $H(c)$ is
similar to a diagonal matrix with positive diagonal elements and
therefore positive definite. We have thus proved that
$c\in\interior{\cal M}_{2n-1}$ and hence the surjectivity of the
Nesterov mapping $N:\interior {\cal M}_{2n-1}\to\interior{\cal P}_{2n-1}$.

\begin{lemma}
If $t_1,\dots, t_{n-1}$ satisfy \eqref{eq17} and $\lambda,\sigma_1,\dots,\sigma_{n-1}$
are given by \eqref{eq.sigrho} (where $p_{2n-2}$ is the leading coefficient of $p$),
then the canonical representation \eqref{eq26} holds.
\end{lemma}

\begin{proof}
We set
\begin{equation} \label{eq.deff}
f_k:=\sum_{j=1}^{n-1}\sigma_jt_j^k+\lambda\delta_{k,2n-2},\qquad k=0,\dots, 2n-2,
\end{equation}
and show that $f_k=c_k$ for all $k$. Note that in view of \eqref{eqPQS}
and $S(t_k)=0$
\begin{equation} \label{eq.Sigma}
\sigma_k=\frac{1}{p(t_k)}=-\frac{1}{p_{2n-2}Q(t_k)S'(t_k)},\qquad k=1,\dots, n-1.
\end{equation}
Setting $x=t_k$ in \eqref{eq37} we get
\begin{equation} \label{eq38}
B(t_k,y)=
\sum_{i=0}^{n-1}\sum_{j=0}^{n-1}b_{ij}t_k^i y^j
=\frac{Q(t_k)S(y)-S(t_k)Q(y)}{t_k-y}
=-Q(t_k)S'(t_k)S_k(y),
\end{equation}
where $S_k$ are the Lagrange interpolation polynomials defined in \eqref{lagrange2}.
Putting $y=t_l$ and using \eqref{eq.Sigma} we obtain
\[
\sum_{i=0}^{n-1}t_k^i p_{2n-2}\sigma_k\sum_{j=0}^{n-1} b_{ij}t_l^j
=\delta_{kl}, \qquad k,l=1,\dots, n-1,
\]
which can be rewritten as
\begin{equation} \label{eq39}
\sum_{j=0}^{n-1}t_l^jy_{jk}=\delta_{kl}, \qquad  k,l=1,\dots, n-1,
\end{equation}
where
\begin{equation} \label{eq.yjk}
y_{jk}:=p_{2n-2}\sigma_k\sum_{i=0}^{n-1}b_{ij}t_k^i, \qquad
j=0,\dots, n-1,\quad k=1,\dots, n-1.
\end{equation}
Since the $S_k$ are polynomials of degree $n-2$, comparing
coefficients of $y^{n-1}$ in \eqref{eq38} yields that
\begin{equation} \label{eq.zero}
\sum_{i=0}^{n-1}b_{i,n-1}t_k^i=0, \qquad k=1,\dots, n-1.
\end{equation}
Defining $y_{in}:=p_{2n-2}b_{i, n-1}$ for $i=0,\dots, n-1$ this can be written as
\begin{equation} \label{eq40}
\sum_{i=0}^{n-1}t_k^i\,y_{in}=0,\qquad k=1,\dots, n-1.
\end{equation}
We also have (see \eqref{eq.yjk} and \eqref{eq.zero})
\begin{equation}\label{eq41}
y_{n-1,k}=p_{2n-2}\sigma_k\sum_{i=0}^{n-1}b_{i,n-1}t_k^i=0,\qquad
k=1,\dots, n-1.
\end{equation}
Since $b_{n-1,n-1}$ is the leading coefficient of the monic
polynomial $H^*(B)=P$ we get
\[
y_{n-1,n}=p_{2n-2}b_{n-1,n-1}=p_{2n-2}=\frac{1}{\lambda}.
\]
This, together with the equations \eqref{eq39}, \eqref{eq40}, and \eqref{eq41},
implies that the matrix $Y$ with entries $y_{ik}$ ($i=0,\dots, n-1$, $k=1,\dots, n$)
is right inverse to the matrix
\[
\begin{bmatrix}
1&t_1&\cdots& t_1^{n-2}&t_1^{n-1}\\
\vdots&\vdots&&\vdots&\vdots\\
1&t_{n-1}&\cdots& t_{n-1}^{n-2}&t_{n-1}^{n-1}\\
0&0&\cdots&0&\lambda
\end{bmatrix}.
\]
Since both matrices are square, $Y$ is also a left
inverse. Inserting the definitions of $y_{ki}$, taking into account that
$b_{jk}=b_{kj}$, and recalling the definition \eqref{eq.deff} of $f_k$,
we find for $k,l=0,\dots, n-1$
\begin{align*} 
\delta_{kl}&=\sum_{i=1}^{n-1}y_{ki}\,t_i^l+\lambda y_{kn}\delta_{l,n-1}
=\sum_{i=1}^{n-1}t_i^l\,p_{2n-2}\sigma_i\sum_{j=0}^{n-1}b_{jk}t_i^j
+\lambda\delta_{l,n-1}p_{2n-2}b_{k,n-1}\\
&=p_{2n-2}\sum_{j=0}^{n-1}b_{kj}\sum_{i=1}^{n-1}\sigma_it_i^{l+j}
+p_{2n-2}\lambda\sum_{j=0}^{n-1}\delta_{l+j,2n-2}b_{kj}\\
&=p_{2n-2}\sum_{j=0}^{n-1}b_{kj}\left(\sum_{i=1}^{n-1}\sigma_it_i^{l+j}
+\lambda\delta_{l+j,2n-2}\right)=p_{2n-2}\sum_{j=0}^{n-1}b_{kj}f_{l+j}.
\end{align*}
Hence the Hankel matrix $H(f)$ is the inverse of the Bezoutian
$p_{2n-2}B=H(c)^{-1}$ and thus $f_k=c_k$ for $k=0,\dots, 2n-2$.  \qed
\end{proof}

\begin{lemma} \label{lem.converse}
Let $c\in \interior {\cal M}_{2n-1}$ be any moment vector such that $N(c)=p$.
If \eqref{eq26} and \eqref{eq28} are satisfied for $x_1,\dots,x_n, t_1,\dots,t_{n-1}$
and positive numbers $\sigma_k,\varrho_k,\lambda$, then also  \eqref{eq10}, \eqref{eq17},
and \eqref{eq.sigrho} hold.
\end{lemma}

\begin{proof}
Since  $H({c} )^{-1}$ is a Bezoutian, there are monic polynomials ${Q}$ and
${S}$ with $H({c} )^{-1}=\tilde{p}_{2n-2}\bez({Q},{S})$ for some constant
$\tilde{p}_{2n-2}\neq 0$ (that will turn out to be the leading coefficient of $p$).
Since the Bezoutian ${B} :=\bez ({Q},{S})$
does not change if we add a multiple of ${Q}$ to ${S}$, we can assume that
$\deg {S}<\deg{Q}\leq n$. Now we have
\begin{equation} \label{eq.pxb}
p(x)=N({c})(x)=H^*(H({c})^{-1})(x) =\tilde{p}_{2n-2}(H^*{B})(x)=\tilde{p}_{2n-2}{B}(x,x).
\end{equation}
Starting from
\[
\tilde{p}_{2n-2}\sum_{j=0}^{n-1}{b}_{kj}{c}_{j+l} =\delta_{kl}, \qquad k,l=0,\dots, n-1,
\]
plugging in
\begin{equation}\label{eq42}
{c}_k=\sum_{j=1}^{n-1}{\sigma}_j{t}_j^k +\delta_{2n-2,k}{\lambda}
=\sum_{j=1}^n{\varrho}_j{x}_j^k,\qquad k=0,\dots, 2n-2,
\end{equation}
and reversing the above computations we arrive at
\begin{alignat}{2}
\label{eq43}
\tilde{p}_{2n-2}\sum_{i=0}^{n-1}\sum_{j=0}^{n-1} {b}_{ij}{x}_{l}^i{x}_k^j
&=\frac{\delta_{kl}}{{\varrho}_k}, \qquad &&k,l=1,\dots, n,\\
\label{eq44}
\tilde{p}_{2n-2}\sum_{i=0}^{n-1}\sum_{j=0}^{n-1} {b}_{ij}{t}_{k}^i{t}_l^j
&=\frac{\delta_{kl}}{{\sigma}_k}, \qquad &&k,l=1,\dots, n-1,\\
\label{eq45}
\sum_{i=0}^{n-1}{b}_{i,n-1}{t}_k^i &=0, \qquad &&k=1,\dots, n-1,\\
\label{eq.lambda}
{\lambda}&=\frac{1}{\tilde{p}_{2n-2}}.
\end{alignat}
By \eqref{eq43}, the polynomial $x\mapsto {Q}(x){S}({x}_k)-{Q}({x}_k){S}(x)$
vanishes for $x=x_l$ and has leading coefficient ${S}({x}_k)$, hence
\[
{Q}(x){S}({x}_k)-{Q}({x}_k){S}(x) ={S}({x}_k)\prod_{l=1}^n(x-{x}_l),
\qquad k=1,\dots, n.
\]
For fixed $x\in\mathbb R$, the polynomial
\[
y\mapsto {Q}(x){S}(y)-{Q}(y){S}(x) -{S}(y)\prod_{l=1}^n(x-{x}_l),
\qquad k=1,\dots, n,
\]
vanishes for $y={x}_k$ and has leading coefficient $-{S}(x)$,
hence we find
\begin{equation}\label{eq46}
{Q}(x){S}(y)-{Q}(y){S}(x)
={S}(y)\prod_{l=1}^n(x-{x}_l)-{S}(x)\prod_{l=1}^n(y-{x}_l).
\end{equation}
By \eqref{eq45}, the coefficient of $y^{n-1}$ in the polynomial
$y\mapsto {B}({t}_k, y)$ vanishes, hence
\[
0=\lim_{y\to\infty}\frac{{B}({t}_k,y)}{y^{n-1}}
=\lim_{y\to\infty}\frac{{S}(y)\prod_{l=1}^n({t}_k-{x}_l)
-{S}({t}_k)\prod_{l=1}^n(y-{x}_l)}{({t}_k-y)y^{n-1}}
={S}({t}_k)
\]
for $k=1,\dots, n-1$, so that
\[
{S}(x)=\prod_{k=1}^{n-1} (x-{t}_k).
\]
By \eqref{eq46} we therefore have
\[
{B}=\bez({Q}_1,{S}), \qquad {Q}_1(x):=\prod_{k=1}^n(x-{x}_k),
\]
and in particular (see \eqref{eq.pxb})
\[
p(x)=\tilde{p}_{2n-2}{B}(x,x)=\tilde{p}_{2n-2} ({S}(x){Q}'_1(x)-{Q}_1(x){S}'(x)).
\]
Hence $\tilde{p}_{2n-2}$ is the  leading coefficient of $p$ and we can write
$\tilde{p}_{2n-2}=p_{2n-2}$. If we put now
\[
{g}(x):=\frac{{Q_1}(x)}{{S}(x)},
\]
and recall that $p(x)=p_{2n-2}P(x)$, we find
\[
{g}'(x)=\frac{ {Q}'_1(x){S}(x) -{Q}_1(x){S}'(x)}{{S}(x)^2} =\frac{P(x)}{{S}(x)^2},
\]
i.e., ${g}$ is a real rational function of the form \eqref{eq11}
having the critical points $\zeta_k,\overline{\zeta}_k$.
Also $s_k>0$ and $a>0$ in \eqref{eq11}  since $g'(x)>0$ and thus $g$ is strictly increasing
in each interval without poles. Using Lemma~\ref{l4} and Lemma~\ref{l5} we conclude that
\eqref{eq17} is satisfied. Consequently, the function $f(x):=-1/g(x)$ is of the form
\eqref{eq1} with $r_k>0$, and has the same critical points $\zeta_k, \overline{\zeta}_k$
as $g$. By Lemma~\ref{l3} and Lemma~\ref{l5}, the equilibrium equation \eqref{eq10} is satisfied.
From equation~\eqref{eq44} with $k=l$ we get 
\[
{\sigma_k}=\frac{1}{p_{2n-2}B(t_k,t_k)} = \frac{1}{p(t_k)},\qquad k=1,\dots, n-1,
\]
and from \eqref{eq43} with $k=l$ we obtain
\[
{\varrho_k}=\frac{1}{p_{2n-2}B(x_k,x_k)} = \frac{1}{p(x_k)},\qquad k=1,\dots, n.
\]
Finally, by \eqref{eq.lambda} we have $\lambda=1/p_{2n-2}$, so that \eqref{eq.sigrho} has
been verified.  \qed
\end{proof}

\begin{lemma}
The Nesterov mapping $N:\interior {\cal M}_{2n-1}\to\interior{\cal P}_{2n-1}$ is  bijective.
\end{lemma}
\begin{proof}
The surjectivity of $N$ has already been shown. To prove its injectivity we assume that
$N(c)=p$ and show that $c\in \interior {\cal M}_{2n-1}$ is unique. By Theorem \ref{Th3}
there are numbers $x_1<t_1<x_2<\dots<t_{n-1}<x_n$ and $\sigma_k,\varrho_k,\lambda>0$
such that
\begin{equation}\label{eq.2canrep}
c_k=\sum_{j=1}^{n-1}\sigma_jt_j^k +\delta_{2n-2,k}\lambda =\sum_{j=1}^n\varrho_jx_j^k,\qquad
k=0,\dots, 2n-2.
\end{equation}
By Lemma \ref{lem.converse}, $(t_1, \dots, t_{n-1})$
is a solution of the equilibrium equation \eqref{eq17}. Since this equation has a unique
solution by Lemma \ref{l6}, we conclude that  $t_1, \dots, t_{n-1}$ are uniquely determined.
Lemma \ref{lem.converse} tells us that \eqref{eq.sigrho} holds, from which we get
unique values of $\lambda$ and $\sigma_1,\dots, \sigma_{n-1}$. Finally, \eqref{eq.2canrep}
determines $c_1, \dots, c_{2n-2}$ uniquely.  \qed
\end{proof}

With the preceding lemma the proof of Theorem \ref{thm4} has been completed.


\section{Concluding remarks} \label{conclusion}

\noindent
In this paper we have established a one-to-one relation between three problems:
\begin{itemize}
\itemsep0mm
\item[{\rm(i)}]
determining a finite Blaschke product from its critical points,
\item[{\rm(ii)}]
finding the equilibrium position of moveable unit charges on the real line
in an electric field generated by a special configuration of negative  point charges,
\item[{\rm(iii)}]
solving the moment problems \eqref{eq26}, \eqref{eq28}.
\end{itemize}
\noindent
Algorithmically, the last problem requires the Vandermonde factorization of the associated
Hankel matrix, but it is difficult to exploit this for solving problems (i)
or (ii). While the transition between (i) and (ii) in both directions is based on simple
transformations,
the translation of (ii) into  (iii) needs the construction of $c\in\interior{\cal M}_{2n-1}$
with $N(c)=p$ for a positive polynomial $p$ with zeros at the given critical points
$\zeta_k,\overline{\zeta}_k$. Hence we have to invert the Nesterov mapping
$N:\interior {\cal M}_{2n-1} \rightarrow \interior {\cal P}_{2n-1}$.
This can be interpreted as finding a positive definite Bezoutian with prescribed anti-diagonal sums
$p\in \interior{\cal P}_{2n-1}$.
Though this problem has a unique solution for every $p\in \interior
{\cal P}_{2n-1}$, we are not aware of an efficient procedure to find it.
On the other hand, $N^{-1}$ can be computed \emph{indirectly},
if problem  (ii) can be solved: starting with (a coefficient vector
of) a positive polynomial $p$, find the corresponding equilibrium positions $t_k$, and
compute $\lambda$ and the weights $\sigma_k$ from  \eqref{eq.sigrho}. The solution
$c\in\interior {\cal M}_{2n-1}$ is then obtained by evaluating  \eqref{eq26}.

We shall discuss these algorithmic and numerical aspects in more detail
in a forthcoming paper.

\bibliographystyle{amsplain}
\bibliography{crit_lib}

\providecommand{\bysame}{\leavevmode\hbox to3em{\hrulefill}\thinspace}
\providecommand{\MR}{\relax\ifhmode\unskip\space\fi MR }
\providecommand{\MRhref}[2]{%
  \href{http://www.ams.org/mathscinet-getitem?mr=#1}{#2}
}
\providecommand{\href}[2]{#2}
\begin{thebibliography}{10}

\bibitem{akr}
N.~I. Ahiezer and M.~Kre\u{\i}n, \emph{Some {Q}uestion in the {T}heory of
  {M}oments}, Translations of mathematical monographs, vol.~2, American
  Mathematical Society, 1962.

\bibitem{bar}
A.~Barvinok, \emph{A {C}ourse in {C}onvexity}, Graduate {S}tudies in
  {M}athematics, vol.~54, American Mathematical Society, 2002.

\bibitem{bou}
T.~Bousch, \emph{Sur quelques probl\`{e}mes de dynamique holomorphe}, Ph.D.
  thesis, Universit\'e Paris 11, Orsay, 1992.

\bibitem{bova}
St. Boyd and L.~Vandenberghe, \emph{Convex {O}ptimization}, Cambridge
  University Press, 2004.

\bibitem{diva}
D.~K. Dimitrov and W.~Van~Assche, \emph{{L}am\'e differential equations and
  electrostatics}, Proc. Amer. Math. Soc. \textbf{128} (2000), 3621--3628.

\bibitem{dim2}
\bysame, \emph{Erratum to "{L}am\'{e} differential equations and
  electrostatics"}, Proc. Amer. Math. Soc. \textbf{131} (2003), no.~7, 2303.

\bibitem{fuh}
P.~A. Fuhrmann, \emph{A {P}olynomial {A}pproach to {L}inear {A}lgebra},
  Springer, 2012.

\bibitem{gorh}
P.~Gorkin and R.C. Rhoades, \emph{Boundary interpolation by finite {B}laschke
  products}, Constr. Approx. \textbf{27} (2008), 75--98.

\bibitem{gri}
A.~Grinshpan, \emph{A minimum energy problem and {D}irichlet spaces}, Proc.
  Amer. Math. Soc. \textbf{130} (2002), no.~2, 453--460.

\bibitem{gru}
F.~A. Gr\"unbaum, \emph{Variations on a theme of {H}eine and {S}tieltjes: an
  electrostatic interpretation of the zeros of certain polynomials}, J. Comput.
  Appl. Math. \textbf{99} (1998), 189--194.

\bibitem{gru1}
\bysame, \emph{Electrostatic interpretation for the zeros of certain
  polynomials and the {D}arboux process}, J. Comput. Appl. Math. \textbf{133}
  (2001), 397--412.

\bibitem{hane}
Y.~Hachez and Yu. Nesterov, \emph{Optimization problems over non-negative
  polynomials with interpolation constraints}, Positive {P}olynomials in
  {C}ontrol (D.~Henrion and A.~Garulli, eds.), Lecture notes in control and
  information sciences, vol. 312, Springer, 2005, pp.~239--271.

\bibitem{hero2}
G.~Heinig and K.~Rost, \emph{Algebraic {M}ethods for {T}oeplitz-like {M}atrices
  and {O}perators}, Mathematical Research, vol.~19, Akademie-Verlag, Berlin,
  1984.

\bibitem{hero1}
\bysame, \emph{Introduction to {B}ezoutians}, Numerical {M}ethods for
  {S}tructured {M}atrices and {A}pplications. The Georg Heinig memorial volume
  (D.~A. Bini et~al., eds.), Operator Theory: Advances and Applications, vol.
  199, Birkh{\"a}user, 2010, pp.~25--118.

\bibitem{hei}
M.~Heins, \emph{On a class of conformal metrics}, Nagoya Math. J. \textbf{21}
  (1962), 1--60.

\bibitem{ioh}
I.~S. Iohvidov, \emph{Hankel and {T}oeplitz {M}atrices and {F}orms}, Nauka,
  Moscow, 1974, Russian.

\bibitem{ism}
M.~E.~H. Ismail, \emph{An electrostatic model for zeros of general orthogonal
  polynomials}, Pacific J. Math. \textbf{193} (2000), 355--369.

\bibitem{kast}
S.~Karlin and W.~J. Studden, \emph{Tschebycheff {S}ystems: with {A}pplications
  in {A}nalysis and {S}tatistics}, Pure and applied mathematics, vol.~15,
  Interscience Publishers, John Wiley \& Sons, New York, 1966.

\bibitem{kr}
D.~Kraus and O.~Roth, \emph{Critical points of inner functions, nonlinear
  partial differential equations, and an extension of {L}iouville's theorem},
  J. Lond. Math. Soc. \textbf{77} (2008), no.~1, 183--202.

\bibitem{kr2}
D.~Kraus and O.~Roth, \emph{Critical points, the {G}auss curvature equation and
  {B}laschke products}, Fields Institute Comm. \textbf{65} (2012), 133--157.

\bibitem{krnu}
M.~G. Kre\u{\i}n and A.~A. Nudel'man, \emph{The {M}arkov {M}oment {P}roblem and
  {E}xtremal {P}roblems}, Translations of {M}athematical {M}onographs, vol.~50,
  American Mathematical Society, 1977.

\bibitem{mmm}
F.~Marcell\'{a}n, A.~Mart\'inez-Finkelshtein, and P.~Mart\'inez-Gonz\'{a}lez,
  \emph{Electrostatic models for zeros of polynomials: Old, new, and some open
  problems}, J. Comput. Appl. Math. \textbf{207} (2007), 258--272.

\bibitem{mara}
A.~Mart\'inez-Finkelshtein and E.~A. Rakhmanov, \emph{Critical measures,
  quadratic differentials, and weak limits of zeros of {S}tieltjes
  polynomials}, Commun. Math. Phys. \textbf{302} (2011), 53--111.

\bibitem{Mash}
J.~Mashreghi, \emph{Derivatives of {I}nner {F}unctions}, Fields {I}nstitute
  {M}onographs, vol.~31, Springer, 2012.

\bibitem{mba}
T.~McMillen, A.~Bourget, and A.~Agnew, \emph{On the zeros of complex {V}an
  {V}leck polynomials}, J. Comput. Appl. Math. \textbf{223} (2009), 862--871.

\bibitem{neh}
Z.~Nehari, \emph{A generalization of {S}chwarz' lemma}, Duke. Math. J.
  \textbf{5} (1946), 118--131.

\bibitem{nest}
Yu. Nesterov, \emph{Squared functional systems and optimization problemes},
  High {P}erformance {O}ptimization (H.~Frenk et~al., eds.), Appl. Optim.,
  vol.~33, Kluwer Academic Publishers, 2000, pp.~405--440.

\bibitem{nene}
Yu.~E. Nesterov and A.~S. Nemirovskii, \emph{Interior-point {P}olynomial
  {A}lgorithms in {C}onvex {P}rogramming}, Studies in Applied Mathematics,
  vol.~13, SIAM Society for Industrial and Applied Mathematics, 1994.

\bibitem{org}
R.~Orive and Z.~Garc\'{i}a, \emph{On a class of equilibrium problems on the
  real axis}, J. Comput. Appl. Math. \textbf{235} (2010), no.~4, 1065--1076.

\bibitem{osl1}
R.~Orive and J.~S\'{a}nchez-Lara, \emph{Equilibrium measures in the presence of
  weak rational external fields}, Preprint. arXiv:1605.01909 [math.{CV}].

\bibitem{osl}
\bysame, \emph{Equilibrium measures in the presence of certain rational
  external fields}, J. Math. Anal. Appl. \textbf{431} (2015), 1224--1252.

\bibitem{ren}
J.~Renegar, \emph{A {M}athematical {V}iew of {I}nterior-point {M}ethods in
  {C}onvex {O}ptimization}, MPS-SIAM series on optimization, SIAM Society for
  Industrial and Applied Mathematics, 2001.

\bibitem{sasu}
D.~Sarason and D.~Suarez, \emph{Inverse problem for the zeros of certain
  {K}oebe-related functions}, Journal d'Analyse Math\'ematique \textbf{71}
  (1997), 149--158.

\bibitem{sewe}
G.~Semmler and E.~Wegert, \emph{Boundary interpolation with {B}laschke products
  of minimal degree}, Comput. Methods Funct. Theory \textbf{6} (2006), no.~2,
  493--511.

\bibitem{shap}
B.~Shapiro, \emph{Algebro-geometric aspects of {H}eine-{S}tieltjes theory}, J.
  Lond. Math. Soc. \textbf{83} (2011), no.~1, 36--56.

\bibitem{ste}
K.~Stephenson, \emph{Introduction to {C}ircle {P}acking: the {T}heory of
  {D}iscrete {A}nalytic {F}unctions}, Cambridge University Press, New York,
  2005.

\bibitem{sti}
T.~J. Stieltjes, \emph{Sur certains polyn\^{o}mes qui v\'erifient une
  \'equation diff\'erentielle lin\'eaire du second ordre et sur la th\'eorie
  des fonctions de {L}am\'e}, Acta Math. \textbf{6} (1885), 321--326.

\bibitem{sze}
G.~Szeg\H{o}, \emph{Orthogonal {P}olynomials}, 4th ed., Amer. Math. Soc.
  Colloq. Publ., vol.~23, American Mathematical Society, 1975.

\bibitem{wan}
Q.~Wang and J.~Peng, \emph{On critical points of finite {B}laschke products and
  the equation $\triangle u=\exp 2u$}, Kexue Tongbao \textbf{24} (1979),
  583--586, Chinese.

\bibitem{zak}
S.~Zakeri, \emph{On critical points of proper holomorphic maps on the unit
  disc}, Bull. London Math. Soc. \textbf{30} (1996), 62--66.

\end{thebibliography}

\end{document}